\documentclass[a4paper,12pt]{article}
\usepackage[left=1in, right=1in, top=1in, bottom=1in]{geometry}
\usepackage[font=small,labelfont=bf,width=0.8\textwidth,skip=16pt]{caption}
\usepackage{graphicx}
\usepackage{here} 
\usepackage{bbm} 
\usepackage{color,latexsym,amsfonts,amssymb,amsmath,mathbbol,bm}
\usepackage{amsthm}
\usepackage{verbatim}
\usepackage{mathrsfs}
\usepackage{fancyhdr}

\numberwithin{figure}{section}
\numberwithin{table}{section}
\numberwithin{equation}{section}

\newcommand{\E}{\mathbb E}
\newcommand{\cF}{\mathcal F}

\newcommand{\cA}{\mathcal A}

\newcommand{\cH}{\mathcal H}

\newcommand{\cL}{\mathcal L}

\newcommand{\Z}{\mathbb Z}

\usepackage{amsthm}
\newtheoremstyle{mythm}{18pt}{0pt}{\itshape}{}{\bfseries}{.}{12pt}{}
\newtheoremstyle{mydefn}{18pt}{0pt}{}{}{\bfseries}{.}{12pt}{}
\theoremstyle{mythm}
\newtheorem{theorem}{Theorem}[section]
\newtheorem{lemma}[theorem]{Lemma}

\newtheorem{corollary}[theorem]{Corollary}
\theoremstyle{mydefn}
\newtheorem{remark}[theorem]{Remark}
\newtheorem{definition}[theorem]{Definition}

\def\Pr{\mathbb{P}} 

\newcommand{\G}{\Gamma}

\def\L{\lambda}

\def\Ref#1{(\ref{#1})}
\parskip = \baselineskip
\newcommand{\eqs}{\begin{eqnarray*}}
\newcommand{\ens}{\end{eqnarray*}}

\def\a{{\alpha}}
\def\b{{\beta}}

\def\d{\delta}

\newcommand{\Lb}{\bm{\lambda}}

\newcommand{\beas}{\begin{eqnarray*}}
\newcommand{\enas}{\end{eqnarray*}}

\newcommand{\eqa}{\begin{eqnarray}}
\newcommand{\ena}{\end{eqnarray}}
\newcommand{\eq}{\begin{equation}}
\newcommand{\en}{\end{equation}}

\def\D{\Delta}
\def\t{\tau}

\def\ignore#1{}

\def\Po{{\rm Pn}}

\def\L{\lambda}
\def\resm{{(m)}}

\def\resone{{(1)}}

\makeatletter
\renewcommand\theequation{\thesection.\@arabic\c@equation}
\makeatother

\usepackage{color,graphicx,mathbbol, enumerate}

\def\ind{\bm{1}}  
\def\E{\mathbb{E}} 

\def\Pr{\mathbb{P}} 
\def\Z{\mathbb{Z}} 

\def\zxab{Z_{\xi+\d_\a + \d_\b}^\resm(t)}
\def\zxabt{Z_{\xi+\d_\a + \d_\b}^\resm}
\def\zx{Z_{\xi}^\resm(t)}
\def\zxt{Z_{\xi}^\resm}
\def\zxa{Z_{\xi+\d_\a}^\resm(t)}
\def\zxat{{Z_{\xi+\d_\a}^\resm}}
\def\zxb{Z_{\xi+ \d_\b}^\resm(t)}

\def\zxaut{Z_{\xi + \d_\a - \d_U}^\resm}
\def\zxu{Z_{\xi - \d_U}^\resm(t)}
\def\zxut{Z_{\xi - \d_U}^\resm}

\def\poml{\Po^\resm(\Lb)}
\def\dot{(\cdot)}

\def\Ref#1{(\ref{#1})}

\begin{document}
\parindent 0cm
\parskip .5cm


\title{Conditional Poisson process approximation}

\author{H. L. Gan\footnote{ Department of Mathematics,
Washington University in St.~Louis,
One Brookings Drive,
St. Louis, MO 63130, USA. E-mail: hgan@math.wustl.edu}
}    
\date{\today}
\maketitle

\begin{abstract} 
Point processes are an essential tool when we are interested in where in time or space events occur. The basic starting point for point processes is usually the Poisson process. Over the years, Stein's method has been developed with a great deal of success for Poisson point process approximation. When studying rare events though, typically one only begins modelling after the occurrence of such an event. As a result, a point process that is conditional upon at least one atom, is arguably more appropriate in certain applications. In this paper, we develop Stein's method for conditional Poisson point process approximation, and closely examine what sort of difficulties that this conditioning entails. By utilising a characterising immigration-death process, we calculate bounds for the Stein factors.
\end{abstract}



\section{Introduction}

In Gan \& Xia~\cite{GX15}, Stein's method for the compound Poisson (and the special cases of Poisson and negative binomial) distribution conditioned upon being greater than $m$ was formulated. This work was initially motivated for the modelling of extreme events such as earthquakes, where the incidence of one extreme event would often lead to many more, hence an understanding of the conditional distribution is often of significant importance. In this paper, we will aim to extend conditional random variable approximation to conditional random process approximation.

To give some motivation for the following work, we introduce the example of Hawkes processes, a class of processes commonly used in financial literature. For a recent survey see Laub, Taimre \& Pollett~\cite{LTP2015}. A Hawkes point process is a point process in time where the incidence rate is in some sense `excited' by other recent points in the process. Intuitively, this is a natural choice of model for earthquakes as given one earthquake has occured, numerous aftershocks typically occur. If we were to approximate such a process with a simple distribution like Poisson, given that there should be a large mass at 0, this approximation would likely be inaccurate. In contrast to such an approach, if we were instead to focus on approximating the distribution given we know that one incident has occurred, we may be able to achieve an accurate approximation. In the case of a Hawkes process, if we know the base/unexcited incidence rate, the probability of zero events occurring is known, and hence understanding the conditional distribution would lead directly to understanding the complete distribution.

Compound Poisson point process approximation theory using Stein's method was first developed by Barbour \& M{\aa}nsson~\cite{BM02}. However, there are many technical difficulties with the approach and the utility of the results are somewhat limited due to the great generality of the compound Poisson distribution. 
In contrast to this, Poisson point process approximation via Stein's method has been far more successful, and as a result, we shall therefore focus on formulating conditional Poisson point process approximation theory.

In this paper we will let $\G$ denote a locally compact complete separable metric space, and $\mathcal{H}$ denote the space of all locally finite point measures on $\G$.

Stein's method for Poisson point process approximation was initiated by Barbour~\cite{barbour88} and Barbour \& Brown~\cite{BB92} as a generalisation of the Stein-Chen method by Chen~\cite{chen75}, and was later refined by Brown, Weinberg \& Xia~\cite{BWX00}, Xia~\cite{xia05} and Xia \& Schuhmacher~\cite{SX08}. The general approach is to utilise a generator for which the associated stationary distribution is a Poisson point process with intensity measure $\Lb$, denoted by $\Po(\Lb)$, and then to use suitable techniques involving couplings to find the relevant bounds. For a given configuration $\xi \in \mathcal{H}$, define the operator $\cA$ on a suitably rich family of functions $h$
\[ \mathcal{A} h(\xi) = \int_\G \left[h(\xi+\d_\a) - h(\xi)\right]\Lb(\d_\a) + \int_\G \left[ h(\xi-\d_\a) - h(\d_\a)\right]\xi(d\a). \]

The operator $\cA$ is the generator of a spatial immigration-death process with immigration rate $\Lb$ on $\Gamma$, unit per capita death rate and the associated stationary distribution to the generator $\mathcal{A}$ is a Poisson process with mean measure $\bm\L$. We now define $h_f(\xi)$ to be the solution to the following (if it exists),
\[ \cA h_f(\xi) = f(\xi) - \Po(\bm{\L})(f),\]
for all $f$ from a suitable family of functions $\mathcal{F}$. Then given a point process $\Xi$, the aim is to use properties of the function $h_f$ to estimate $|\E f(\Xi) - \Po(\bm\L)(f)|$ by finding a bound for $|\cA h_f(\Xi)|$.

We will use the term \emph{conditional} to mean conditional upon at least $m$ atoms in the entire space $\Gamma$. $\Xi$ follows the distribution of a \emph{conditional Poisson point process}, if it has the distribution of a Poisson point process conditional on having at least $m$ atoms in $\Gamma$. Its distribution will be denoted by $\Po^\resm(\Lb)$. Note in general we will use the notation $\vphantom{a}^{(m)}$ to denote conditional upon $m$ atoms.


We will now associate a conditional Poisson point process with the limiting distribution of a spatial immigration-death process, in a similar fashion to how it is defined for conditional random variable approximation in Gan \& Xia~\cite{GX15}. Given the process is in configuration $\xi$, with $|\xi| \geq m$, then the process will stay at this configuration for an exponentially distributed time with mean $\frac{1}{|\xi|\ind_{|\xi|>m} + \L}$, where $\L = \Lb(\G)$. With probability $\frac{\L}{|\xi|\ind_{|\xi|>m} + \L}$, the process will then add a new point into the system at a point following distribution $\Lb/\L$ , or with probability $\frac{|\xi|\ind_{|\xi|>m} }{|\xi|\ind_{|\xi|>m} + \L}$, one of the existing points chosen uniformly at random will be removed. The generator of such a process is
\begin{align}
\cA^\resm h(\xi) = \int_\G \left[h(\xi+\d_\a) - h(\xi)\right]\Lb(\d_\a) + \int_\G \left[ h(\xi-\d_\a) - h(\d_\a)\right]\xi(d\a) \cdot \ind_{|\xi| > m}. \label{generator}
\end{align}

\begin{lemma}
The unique stationary distribution for the generator $\cA^\resm$ is $\Po^\resm(\Lb)$.
\end{lemma}
\begin{proof} We can apply Theorem~7.1 from Preston~\cite{preston75} to show the existence and uniqueness of the stationary distribution for $\cA^\resm$. It remains to show that if $\Xi \sim \Po^\resm (\Lb)$ then $\E \cA^\resm h(\Xi)= 0$, which can be verified via a direct calculation.
\end{proof}
It should be noted that we can also censor the immigration rate at a level $n>m$ and we would then have a Poisson point process conditioned on having a number of atoms between $m$ and $n$ as the stationary distribution. In line with the original motivation, for this paper we will just focus on conditioning from below.

As per usual in Stein's method, for any bounded function $f$, we set up a Stein equation, and hope to solve for a $h^\resm_f(\xi)$ that satisfies
\begin{align}
 \cA^\resm h_f^\resm(\xi) = f(\xi) - \Po^\resm(\Lb)(f),\label{CPPsteineq}
\end{align}
where $\Po^\resm(\Lb)(f) = \E f(Z^\resm)$ and $Z^\resm \sim \Po^\resm(\Lb)$.
\begin{lemma}
For all bounded functions $f$,
\begin{align*}
h_f^\resm(\xi)= -\int_0^\infty \E \left[ f(\zx - \poml(f)\right]dt,
\end{align*}
where $\zxt\dot$ is a spatial immigration-death point process with generator $\cA^\resm$ and $Z_\xi^\resm(0) = \xi$, is well defined and is the solution to the Stein equation \Ref{CPPsteineq}.
\end{lemma}
Our metric of choice to evaluate the distance between two point processes will be the $\overline{d}_2$ metric first introduced in Xia~\cite{xiathesis} and systematically studied in Schuhmacher \& Xia~\cite{SX08}. Important to note is that the metric $\bar{d_2}$ encapsulates convergence in distribution of point processes, see Proposition~2.3 of Schuhmacher \& Xia~\cite{SX08}. 

\begin{definition}
For $\xi = \sum_{i=1}^n \d_{x_i}$, $\eta = \sum_{i=1}^m \d_{y_i} \in \cH$ with $n \geq m$, and $d_0\leq 1$ as the metric on $\Gamma$, the metric $\bar{d}_1$ is defined by
\[ \overline{d}_1(\xi, \eta) = \frac{1}{n} \left( \min_{\pi \in \Pi_n} \sum_{i=1}^m d_0(x_{\pi(i)}, y_i) + (n-m) \right),\]
where $\Pi_n$ is the set of all permutations of $\{1, \ldots, n\}.$
\end{definition}
\begin{definition}
Let $\cF_{\bar{d}} = \{ f : \cH \rightarrow [0,1]: |f(\xi) - f(\eta)| \leq \overline{d}_1(\xi,\eta)\ \forall\ \xi,\eta \in \cH\}$. Then for any two point distributions $P$, $Q$ on $\cH$, define
\[\overline{d}_2(P,Q) = \sup_{f \in \cF_{\bar{d}}} \left| \int f dP  - \int f dQ \right|.\]
\end{definition}

In this paper, we will assume that $\Lb$ is \emph{diffuse}, that is, it has no atoms in $\G$. If $\Lb$ is not diffuse, we can approximate it by a diffuse measure accurately by \emph{lifting} the process as described in Chen \& Xia~\cite{CX04}. Furthermore for any configuration $\xi$ we will also without loss of generality assume that the points in $\xi$ are all distinct. Similarly to how we can assume $\Lb$ is diffuse, for any non-distinct points in a configuration $\xi$ we can `shift' points by a small amount and then take limits due to the continuity of the metric $\bar{d}_1$.

To succesfully apply Stein's method, we typically require over the family of functions $f \in \cF_{\bar{d}}$, bounds for:
\begin{align}
\| \D h^\resm \| &:= \sup_{\xi} \|\D h_f^\resm(\xi;\a)\| :=\sup_{\xi} \sup_{f,\a} | h_f^\resm(\xi + \d_\a) - h_f^\resm(\xi)|,\label{D1}\\
\| \D^2 h^\resm \| &:= \sup_{\xi} \|\D^2 h_f^\resm(\xi;\a,\b)\|\notag\\
&:=\sup_{\xi}\sup_{f,\a,\b} |  h_f^\resm(\xi + \d_\a + \d_\b) - h_f^\resm(\xi+ \d_\a)  - h_f^\resm(\xi + \d_\b) + h_f^\resm(\xi)|.\label{D2}
\end{align}
\begin{theorem}\label{firstdiff}
Define
\[K_1 :=  \min \left( \frac{1}{m}, \frac{0.95 + \log^+\L}{\L} \right).\]
For $m \geq 1$,
\[\|\D h^\resm\| \leq \frac{1}{\L} + (m+1)K_1,\]
and if $\L > m+2$,
\[\|\D h^\resm\| \leq\frac{1}{\L(\L - m)} + \frac{\L}{\L - m} K_1,\]
\end{theorem}
\begin{theorem}\label{seconddiff}
Define
\[ K_2 :=\frac{2\log \L}{\L} \ind_{\L \geq 1.76} + \frac{1}{m+1} \ind_{\L < 1.76}. \]
For $m \geq 1$,
\begin{align*}
 \|\D^2 h^\resm\| \leq &\min\Big\{\frac{2}{\L} + 2(m+1)K_1,\\
&\frac{(4m+3)(m+3)}{(m+3)(2m+2)\lambda + 2\lambda^2} + \frac{4m(m+1)(m+3)}{(m+3)(2m+2) + 2\lambda} K_1+K_2\Big\},
\end{align*}
and if $\L > m+2$,
\begin{align*}
\|\D^2 h^\resm\| \leq \min\left\{ \frac{2}{\L(\L - m)} + \frac{2\L}{\L - m}K_1,\frac{3\L + m}{\L(\L-m)(\L+m)} + \frac{4\L m}{(\L-m)(\L+m)}K_1 + K_2\right\}.
\end{align*}
\end{theorem}

This paper will be set out in the following manner. Section 2 will focus on proving the above bounds for the Stein factors, this will also include a short diversion into some non-uniform bounds for the first and second difference of $h$, and section 3 will include a short application. 

\section{Bounds for the Stein factors}
\subsection{Bounds for the first difference of $h$}
When deriving bounds for unconditional Poisson point process approximations, the canonical technique to calculate bounds for \Ref{D1} and \Ref{D2} is to couple the additional points at $\a$ and $\b$ independently to $Z_\xi \dot$. More precisely, we can set
\begin{align*}
Z_{\xi + \a}(t) &= Z_{\xi}(t) + \d_\a \ind_{\tau_\a > t},\\
Z_{\xi + \a + \b}(t) &= Z_{\xi}(t) + \d_\a \ind_{\tau_\a > t} + \d_\b \ind_{\tau_\b > t},
\end{align*}
where $\tau_\a, \tau_\b$ are independent exponential random variables with mean 1, and are independent of $Z_\xi \dot$. As a result, the bounds for \Ref{D1} and \Ref{D2} will generally depend upon how long it takes for the particles at $\a$ and $\b$ to die. In our conditional setting, one would hope we would be able to similarly `separate' the extra point at $\a$ from $\xi$, and run a (not necessarily independent) pure death process for the point at $\a$. In the case of conditional Poisson approximation, this approach works, but it does not for point processes. The problem is created by the location of deaths, where as in random variable approximation we essentially only care about the total number. The following reveals the problems that keeping track of particle locations generates.

Suppose $m=1$, we would like to define a coupling such that:
\begin{align} Z_{\xi + \d_\a}^\resone(t) \stackrel{d}{=} Z_\xi^\resone(t) + \d_\a \ind_{\tau_\a>t},\label{approx}\end{align}
for some stopping time $\tau_\a$. However, such a coupling does not exist because of the following reason. $Z_{\xi + \d_\a}^\resone(\cdot) = \d_\a$ is a configuration that can be achieved with a positive probability, but it is not a possible realisation for $Z_\xi^\resone(\cdot) + \d_\a \ind_{\tau_\a>\cdot}$, as $|Z_\xi^\resone(\cdot)| \geq 1$. To deal with the lack of a direct coupling, we shall formulate an `approximate' coupling that enables us to use the formulation of the right hand side of \Ref{approx} in the following manner
\begin{align}
\left|\D h^\resm_f(\xi;\a)\right| &= \left| \int_0^\infty \left[ \E f(\zxa) - \E f(\zx) \right] dt \right|\notag\\
	&\leq \left| \int_0^\infty \left[ \E f(\zxa) - \E f(\zx + \d_\a \ind_{\tau_a>t}) \right] dt \right|\notag\\
	&\phantom{V} +  \left| \int_0^\infty \left[ \E f(\zx + \d_\a \ind_{\tau_a>t}) - \E f(\zx) \right] dt \right|.\label{split}
\end{align}
An important question is, how do we best define $\tau_\a$ so as to minimise these two integrals in \Ref{split}? Examining the first integral of \Ref{split}, due to our choice of metric, it would be best if we defined $\tau_\a$ such that the two processes $\zxat \dot$ and $\zxt \dot + \d_\a \ind_{\tau_\a > \dot}$ both have the same number of particles at all times. To this end, we can define $\tau_\a$ such that 
\begin{align}
\Pr\left(\tau_\a > t\left| |Z_\xi^\resm\dot|\right)\right. = \exp\left\{-\int_0^t (1 + m \ind_{|Z_{\xi}^\resm(s)| = m}) ds\right\}. \label{taua}
\end{align} In essence, given $\zxat \dot$ and $\zxt \dot + \d_\a \ind_{\tau_\a > \dot}$ have more than $m+1$ particles, then they can be coupled exactly. If the number of particles reaches $m+1$ and the point at $\a$ is still alive, then in $\zxt \dot + \d_\a \ind_{\tau_\a > \cdot}$, all the death rates are forced to the single point at $\a$.

The decomposition of \Ref{split} may initially seem to be somewhat arbitrary, however there is an interpretation for this decomposition if we think about $|\D h_f^\resm(\xi;\a)|$ in the following manner. Essentially what needs to be considered is how long it takes for $\zxat \dot$ and $\zxt \dot$ to coalesce, and the aim is to find a coupling so this occurs as quickly as possible. Naively, one might think that we are simply waiting for the point at $\a$ to die, and until it dies, $\zxat \dot$ will have exactly one more point than $\zxt\dot$. 
As long as $|\zxat\dot| > m+1$, we can couple of the two processes exactly. The problem occurs if we reach a state where $|\zxat\dot| = m+1$ and the point at $\a$ is still alive. At this time, $\zxat\dot$ would still have per capita death rate, but as $\zxt\dot$ would only have $m$ particles, it would therefore have a net death rate of 0. From here it is possible that a particle that is not $\a$ will die from $\zxat\dot$, and hence we need to account for this. 

The decomposition in \Ref{split} is designed to address this issue exactly. Notice that only one of the integrals in \Ref{split} is ever non-zero at any given time. The term in the second integral is going to be $0$ for $t > \tau_\a$. Similarly, the term in the first integral is $0$ until $t > \tau_\a$, and then from this point onwards, it may be non-zero. In this manner, we can see that the second integral essentially takes care of $|\D h_f^\resm(\xi;\a)|$ until an additional death occurs in $\zxat\dot$ compared to $\zxt\dot$, and the first integral accounts for the chance that the additional death may not be the point at $\a$. Before we prove Theorem~\ref{firstdiff} we will need a number of lemmas.

\begin{lemma}\label{zmz0}
There exists a coupling such that
\[ | Z_\xi^\resm \dot| \geq |Z_\xi\dot| \geq |Z_0\dot|,\]
where $Z_\xi \dot:= Z_\xi^{(0)}\dot$, and $Z_0\dot$ is a process that follows generator $\mathcal{A}^{(0)}$ with $Z_\xi(0) = \emptyset.$
\end{lemma}
\begin{proof}
For $|Z_\xi\dot|$, the birth and death rates are, $\forall i \in \{ 0,1,\ldots \}$,
\begin{align*}
\a_i = \L ,\ \ 
\b_i = i.
\end{align*}
For $|\zxt\dot|$, the birth and death rates are, $\forall i \in \{ 0,1,\ldots \}$,
\begin{align*}
\a_i = \L, \ \ 
\b_i = \begin{cases}
0 & i = m,\\
i & \text{otherwise.}
\end{cases}
\end{align*}

From the above, it is clear that the birth rates of both processes match, and the death rates for $|Z_\xi\dot|$ are greater than those for $|\zxt\dot|$. Intuitively, as $|\zxt\dot|$ has strictly lower death rates, it should stochastically dominate $|Z_\xi\dot|$. To rigorously show the required result, we can use the coupling from Lindvall~(p.~163)~\cite{lindvall92}.

For the second inequality, it is sufficient to note that we can use a coupling to define
\begin{align*}
Z_\xi(t) = Z_0(t) + D_\xi(t),
\end{align*}
where $D_\xi(t)$ is a pure-death process with unit per capita death rate independent of $Z_0(t)$. For details of this coupling, see Proposition~3.5 in Xia~\cite{xia05}. 
\end{proof}


\begin{lemma}\label{alphadeath}
If $|\xi| = k\geq m$, $\xi(\{\a\})=0$ and $\tau_{k} = \inf \{ t : |\zxa| = k\}$, then
\[p_{\L,k} := \Pr\left(Z_{\xi + \d_\a}^\resm(\tau_{k}) (\{\a\}) = 1\right) \leq \min \left(\frac{k}{\L}, \frac{k}{k+1}\right).\]
\end{lemma}
\begin{proof}
First, we note that as the process has more than $m$ particles for all $t < \tau_k$, each particle can be treated independently over this time period. As a result, it suffices to consider $S$, the number of original points from the configuration $\xi + \d_\a$ that still remain in the system at time $\tau_k$. Due to the independence, given $S$ we then know that each original point is equally likely to survive with probability $\frac{S}{k+1}$, and we can use this fact to calculate $p_{\L,k}$. 

If we let $N = k-S$, be the number of new points in the system at time $\tau_{k}$ that have originated from immigration, it suffices to study $N$, and furthermore $N$ has the same distribution as that of $|Z_0(\tau_{k})|$, where $Z_0\dot$ is the process that tracks immigrants and their deaths from the process $\zxat\dot$. Therefore conditioning upon $|Z_0(\tau_{k})|$,
\begin{align*}
p_{\L,k} &= \sum_{i=0}^{m}\Pr\left( Z_{\xi + \d_\a}^\resm(\tau_{k}) (\{\a\}) = 1 \Big{|} |Z_0(\tau_{k})| = i\right) \Pr(| Z_0(\tau_{k})| = i).
\end{align*}
If at time $\tau_k$ there are $i$ particles in our system that arrived by immigration, and $k$ particles in total at time $\tau_{k}$, then there must be $k-i$ surviving original points in the system, and therefore
\begin{align*}
p_{\L,k} &= \sum_{i=0}^k \left( \frac{k - i}{k+1} \right)\cdot \Pr( |Z_0(\tau_{k})| = i)\\
	&=1 - \frac{1}{k+1} \E (|Z_0(\tau_{k})|+1).
\end{align*}
Brown \& Xia~\cite{BX01} (Eq.~5.17), showed that $$\E | Z_0(\tau_{k})| = -1 + (k+1) \left( \frac{\bar{F}(k-1)}{\bar{F}(k)} - \frac{k}{\L} \right),$$
where $\bar{F}(i) = \sum_{j=i}^\infty \Po(\L)(\{j\})$. Therefore
\begin{align*}
p_{\L,k} &= 1 - \left( \frac{\bar{F}(k-1)}{\bar{F}(k)} - \frac{k}{\L} \right) \leq \frac{k}{\L}.
\end{align*}
The above bound is redundant if $\L < k$. To achieve the $\L$-independent bound, we consider the last transition of the process at time $\tau_{k}$. Assuming the point at $\a$ is alive at this time, the probability of the point at $\a$ surviving the final death event is $\frac{k}{k+1}$. Therefore if $\L > k+1$, we use $\frac{k}{\L}$ as our bound, otherwise we can use $\frac{k}{k+1}$. This bound can not be improved in general as for $m=0$ and $k=0$ or $\L = 0$, equality holds.
\end{proof}

\begin{lemma}\label{lemmaI1}
For $m \geq 1$, define
\begin{align*}
I_1(\xi) := \int_0^\infty \E \left[ f(Z_{\xi + \d_\a - \d_U}^\resm(t)) - f(\zx) \right] dt,
\end{align*}
where $U$ is chosen uniformly at random from $\xi$. Let $|I_1| := \sup_{\xi : |\xi|\geq m} | I_1(\xi)|$, then
\begin{align*}
|I_1| &\leq (m+1) \left[ \frac{1}{\L m} + K_1 \right].
\end{align*}
Furthermore, if $\L > m+1$,
\begin{align*}
|I_1| &\leq \frac{1}{m(\L-m)} + \frac{\L}{\L-m}\cdot K_1.
\end{align*}
\end{lemma}
\begin{proof}
If we pair the death times of $U$ in $\zxt\dot$ with $\a$ in $\zxaut\dot$, then both $\zxaut\dot$ and $\zxt\dot$ can be coupled in such a way that their birth times and death times are matched. Note that we get a complete coupling, i.e. the two processes become identical, when the particle at $\a$ and the corresponding $U$ leave the systems. Noting that the time of death for the particle at $\a$, $T_\a$, satisfies $\Pr\left(T_\a > t \big| |\zxu|\right) = \exp \{ -\int_0^t (1 - \ind_{|\zxut(s)| = m-1} )ds \}$, it can be seen that as the death rate for the particle at $\a$ `turns off' when there are $m$ particles in the system, if $\xi$ is larger, then the point at $\alpha$ is more likely to die sooner. Therefore it suffices to consider the worst case scenario $|\xi| = m$. Define
\begin{align*}
S &:= \inf\{ t : |Z_{\xi + \d_\a - \d_U}^\resm(t)| = m+1\}\sim \exp(\L),\\
T &:= \inf\{ t : |Z_{\xi + \d_\a - \d_U}^\resm(t)| = m, t > S\}.
\end{align*}
Using the strong Markov property, 
\begin{align}
&\left| \int_0^\infty  \left[ \E f(Z_{\xi + \d_\a - \d_U}^\resm(t)) - f(\zx) \right] dt \right| \notag\\
&= \left| \E \int_0^{T_\a}  \left[  f(Z_{\xi + \d_\a - \d_U}^\resm(t)) - f(\zx) \right] dt \right| \notag\\
	&\leq \E \int_0^{S}\left| f(Z_{\xi + \d_\a - \d_U}^\resm(t))-  f(\zx)\right|dt\notag\\
	&\ \ +\E \int_{S}^{T} \left|  f(Z_{\xi + \d_\a - \d_U}^\resm(t)) -  f(\zx) \right| \ind_{T_\a > t}\ dt  + \Pr(T_\a > T) |I_1|.\label{thing}
\end{align}
For the first integral of \Ref{thing}, recalling the choice of $\cF_{\bar{d}}$ for our metric,
\begin{align*}
&\E \int_0^{S} \left| f(Z_{\xi + \d_\a - \d_U}^\resm(t))-  f(\zx) \right| dt \leq \E \int_0^{S} \frac{1}{m} dt= \frac{1}{m} \E S = \frac{1}{\L m}.
\end{align*}
For the second integral in \Ref{thing},
\begin{align*}
&\E \int_{S}^{T} \left|  f(Z_{\xi + \d_\a - \d_U}^\resm(t)) -  f(\zx) \right| \ind_{T_\a > t}\ dt\\
	& = \E \int_{S}^{T} \left|  f(Z_{\xi - \d_U}^{(m-1)}(t) + \d_\a) -  f(Z_{\xi - \d_U}^{(m-1)}(t) + \d_U) \right| \ind_{T_\a > t}\ dt\\
	&\leq \E \int_{S}^T \frac{1}{|Z_{\xi - \d_U}^{(m-1)}(t)|+1} \cdot e^{-(t-S)} dt \\
	&\leq \E \int_{S}^\infty \frac{1}{|Z_{\xi - \d_U}^{(m-1)}(t)|+1} \cdot e^{-(t-S)} dt \\
	&\leq \int_0^\infty e^{-t} \E \Bigg{[}\frac{1}{|Z_{Z_{\xi - \d_U}^{(m-1)}(S)}^{(m-1)*}(t)|+1}\Bigg{]}dt \leq K_1,
\end{align*}
where we have used the strong Markov property, $Z_\xi^{(m-1)*}\dot$ is an independent process that follow generator $\cA^{(m-1)}$, Lemma~\ref{zmz0} in the last inequality, and that the death rate for the particle $\a$ is 1 in the time interval $[S,T]$. Noting that that we have $\Pr(T_\a > T) =p_{\L,m}$, using Lemma~\ref{alphadeath} in the following equation yields the result.
\[ |I_1| \leq \frac{1}{\lambda m} + K_1 + p_{\lambda,m}|I_1|.\]
\end{proof}
\begin{lemma}\label{lemmaI2}
For $m \geq 1$, define
\[I_2(\xi) := \int_0^\infty \left[ \E f(\zxa) - \E f(\zx + \d_\a \ind_{\tau_\a>t})\right]dt,\]
where $\tau_\a$ is defined as in \Ref{taua}. Let $|I_2| := \sup_{\xi : |\xi| \geq m} |I_2(\xi)|$, then
\begin{align*}
	|I_2| \leq \frac{1}{\L} + mK_1.
\end{align*}
If $\L > m+2$, then
\begin{align*}
|I_2| &\leq \frac{1}{\L(\L - m)} + \frac{m}{\L - m} K_1.
\end{align*}
\end{lemma}
\begin{proof}
As long as $|\zxat\dot| > m+1$ and $|\zxt\dot + \d_\a \ind_{\tau_\a>\cdot}| > m+1$, the death rates for both processes at all points are identical, and we can couple the two processes identically. A problem occurs if we reach the state where there are only $m+1$ points in both systems, and the point at $\a$ is still alive. In $\zxat\dot$, each particle has per capita death rate, but in $Z_{\xi + \d_\a}^\resm\dot + \d_\a \ind_{\tau_\a>\cdot}$, only the point at $\a$ has death rate $m+1$. This scenario is where the two processes could diverge. 



Similarly to Lemma~\ref{lemmaI1} it suffices to consider the case $|\xi| = m$. We decompose $|I_2(\xi)|$ by considering the first transition of both processes. Our final coupling time will be when the point at $\a$ is dead in both processes. We can define a coupling such that
\begin{itemize}
\item With probability $\frac{1}{m+1+\L}$ the first transition will be a death and $\a$ will be the point chosen to die from both $\zxat\dot$ and $Z_{\xi}^\resm\dot + \d_\a \ind_{\tau_\a>\cdot}$.
\item With probability $\frac{m}{m+1+\L}$ the first transition will be a death, and $\a$ in $Z_{\xi}^\resm\dot + \d_\a \ind_{\tau_\a>\cdot}$ will die, but a uniformly selected point $U$ from $\xi$ of $\zxat\dot$ will die.
\item With probability $\frac{\L}{m+1+\L}$ the first transition will be an immigration. 
\end{itemize}
If the first transition is immigration, we can couple the points of two processes exactly until the return to $m+1$ particles. If the point at $\a$ dies before the return to $m+1$ particles, then the coupling is complete, if not, then we essentially return to the initial starting state.  Therefore,
\begin{align}
|I_2(\xi)| &\leq \frac{m}{m+1+\L} \left| \int_0^\infty \left[ \E f(Z_{\xi + \d_\a - \d_U}^\resm(t)) - f(\zx) \right] dt \right| + \frac{\L}{m+1+\L} \cdot p_{\L,m+1} |I_2|,\label{I2bit}
\end{align}
where $p_{\L,m+1}$ is defined as in Lemma~\ref{alphadeath}. Using Lemma~\ref{alphadeath} and noting that the integral in \Ref{I2bit} can be bounded by $|I_1|$ in Lemma~\ref{lemmaI1}, some rearrangement yields the lemma.
\end{proof}
We are now ready to complete the proof of Theorem \ref{firstdiff}.
\begin{proof}[Proof of Theorem \ref{firstdiff}]
Lemma \ref{lemmaI2} accounts for the first half of \Ref{split}. For the second integral of \Ref{split}, 
\begin{align*}
&\left|\int_0^\infty \left[ \E f(\zx + \d_\a \ind_{\tau_\a>t}) - \E f(\zx) \right] dt\right|\\
&= \left|\int_0^\infty \E \left[ f(\zx + \d_\a) - f(\zx) \middle| \tau_\a > t \right] \Pr(\tau_\a > t) dt\right|\\
&\leq \int_0^\infty\E \left[ \frac{1}{|Z_\xi^{(m)}(t)|+1} \right] e^{-t} dt \leq K_1,
\end{align*}
where $\tau_\a$ is defined as in \Ref{taua}. 
The first bound in $K_1$ is true as $|Z_\xi^{(m-1)}(t) + 1| \geq m$. For the other bound, using Lemma~\ref{zmz0},
\begin{align*}
\E \left[ \frac{1}{|Z_\xi^{(m-1)}(t)| + 1} \right] \leq \E \left[ \frac{1}{|Z_0(t)| + 1} \right],
\end{align*}
and the bound for $\int_0^\infty e^{-t} \E  \left[ \frac{1}{|Z_0(t)| + 1} \right] dt$ can be found in Schuhmacher \& Xia~\cite{SX08}. Putting everything together we achieve the final bound.
\end{proof}
\begin{remark}\label{sameasuncond}
In the proof of Lemma~\ref{lemmaI2}, if $m=0$ and $|\xi| = 0$, notice that the probability of the first transition being a death, but $\a$ not dying from $Z_{\xi + \d_\a}^{(0)}$ is 0. Therefore, the processes never diverge and hence $|I_2| = 0$. Furthermore, this implies $\|\D h^{(0)}\| \leq K_1$ (albeit with the small modification of using $1$ for the constant term instead of $\frac{1}{m}$), consistent with the unconditional bounds of Schuhmacher \& Xia~\cite{SX08}.
\end{remark}

\begin{remark}\label{sameasuncond2}
It is worth comparing the above bound to bounds in the unconditional case. Proposition~4.1 from Schuhmacher \& Xia~\cite{SX08} gives
\begin{align}
\|\D h^{(0)} \| \leq \min\left(1, \frac{0.95 + \log^+\L}{\L}\right). \label{SXfirstdiff}
\end{align}
Therefore our bound in the conditional case is generally slightly worse than the bounds in the unconditional case. However, for large $\L$, the bound is asymptotically the same as $K_1$, so it appears that the additional term is not too bad as long as $\L$ is of a reasonable size. 
\end{remark}

When $\L$ is small, the bounds in Theorem~\ref{firstdiff} are large and therefore may not be useful in applications. In the unconditional scenario, when $\L$ is small, the constant bound of $1$ is used. An important question is, does there exist a $\L$-independent bound for $\|\Delta h^\resm\|$ like in the unconditional case? The answer to this appears to be no. Consider
\begin{align*}
\left| \int_0^\infty \E \left[ f(\zxa) - f(\zx) \right] dt \right|,
\end{align*}
with a starting configuration such that $\left|\xi\right| = m$ and $\L$ is very small. The first transition for $|\zxat\dot|$ is going to be a death with probability $\frac{m+1}{m+1+\L}$, or an immigration with probability $\frac{\L}{m+1+\L}$. Given that $\L$ is small, this implies that the first transition is almost certainly going to be a death, and with probability $\frac{m}{m+1}$ the particle chosen for death is not going to be $\a$. Meanwhile $\zxt\dot$ is going to be unchanged with high probability as the only possible transition is an immigration step upwards which occurs with rate $\L$. We are therefore likely to reach a state where the processes are differing by a single pair of particles, but the expected time until the next transition is exactly the expected time until the next immigration, $\frac{1}{\L}$. Hence, it appears there will unavoidably be a $\L$-dependent component in the bound. This problem does not occur in the unconditional case as the death rate for the point at $\a$ is always 1, it does not `turn off' as it does in our conditional scenario.


\subsection{Bounds for the second difference of $h$}\label{CPPPseconddiff}
We now seek to extend our approximate coupling idea to calculate bounds for $\|\D^2h^\resm\|$? Again, we would like to use the canonical couplings of the form
\begin{align*}
\zxab &\stackrel{d}{=} \zx + \d_\a \ind_{\tau_\a > t} + \d_\b \ind_{\tau_\b > t},\\
\zxa &\stackrel{d}{=} \zx + \d_\a \ind_{\tau_\a > t},\\
\zxb &\stackrel{d}{=} \zx + \d_\b \ind_{\tau_\b > t},
\end{align*}
for some waiting times $\tau_\a$ and $\tau_\b$, and again the conditioning induces complications. In the unconditional case, we can define $\tau_\a$ and $\tau_\b$ as two independent exponential random variables with rate 1. For exactly the same reasons outlined in the previous subsection, there do not exist any waiting times $\tau_a, \tau_b$ that satisfy the coupling above. As a result, we instead approach the problem by again using `approximate' couplings from before to enable us to maintain some semblance of independence.

Given the filtration of $|\zxt\dot|$, we choose to define $\tau_\a$ and $\tau_\b$ such that $\tau_\a$ and $\tau_\b$ are conditionally independent copies of the waiting time with distribution as in \Ref{taua} . This approach presents not only similar complications as discussed earlier for the first difference, but also gives different net death rates for $\zxabt\dot$ when compared to $\zxt\dot + \d_\a \ind_{\tau_\a > \cdot} + \d_\b \ind_{\tau_\b > \cdot}$. Consider
\begin{align}
&\int_0^\infty \E \left[ f(\zxab) - f(\zxa) - f(\zxb) + f(\zx) \right] dt\notag\\
	&= \int_0^\infty \E \left[ f(\zxab) - f(\zx + \d_\a \ind_{\t_\a > t} + \d_\b \ind_{\t_\b > t}) \right.\notag\\
	&-  f(\zxa) + f(\zx + \d_\a \ind_{\t_\a > t}) - f(\zxb) + f(\zx + \d_\b \ind_{\t_\b > t}) \Big] dt\notag\\
	&+ \int_0^\infty \E \left[ f(\zx + \d_\a \ind_{\t_\a > t} + \d_\b \ind_{\t_\b > t}) - f(\zx + \d_\a \ind_{\t_\a > t})\right.\notag\\
	&\phantom{VVVVV}- \left. f(\zx + \d_\b \ind_{\t_\b > t}) + f(\zx) \right] \label{fourparts}dt.
\end{align}
In \Ref{fourparts}, the absolute value of last integral can be shown to be bounded by existing unconditional bounds, similar to the way the second half of \Ref{split} was bounded by its unconditional equivalent. Therefore, the main work is to find a bound for the first integral of \Ref{fourparts}.

\begin{lemma}\label{lemmaI3}
For $m \geq 1$, define,
\begin{align}
I_3(\xi) :=& \int_0^\infty \E \left[ f(\zxab) - f(\zx + \d_\a \ind_{\t_\a > t} + \d_\b \ind_{\t_\b > t}) \right.\notag\\
	&-  f(\zxa) + f(\zx + \d_\a \ind_{\t_\a > t}) \notag\\
	&- f(\zxb) + f(\zx + \d_\b \ind_{\t_\b > t}) \Big] dt\label{six}
\end{align}
where $\tau_\a$ and $\tau_\b$ are defined as in~\Ref{taua}. Let $|I_3| := \sup_{\xi:|\xi| \geq m} |I_3(\xi)|$, then
\begin{align*}
|I_3| &\leq \frac{(4m+3)(m+3)}{(m+3)(2m+2)\lambda + 2\lambda^2} + \frac{4m(m+1)(m+3)}{(m+3)(2m+2) + 2\lambda} K_1.
\end{align*}
Furthermore, if $\L > m+2$,
\begin{align*}
|I_3| & \leq \frac{3\L + m}{\L(\L-m)(\L+m)} + \frac{4\L m}{(\L-m)(\L+m)}K_1.
\end{align*}
\end{lemma}
\begin{proof}
To begin, note as before, it suffices to assume that $|\xi|=m$ as the worst case scenario. We will decompose $|I_3|$ by conditioning upon the first transition of the processes and assume $|\xi| = m$.

The major complication is that the net death rates of the first transition are not the same for the six processes in \Ref{six}. For the six processes the net death rates are $m+2,\ 2m+2,\ m+1,\ m+1,\ m+1,\ m+1$ in the order given by \Ref{six}. To deal with this, we define the following coupling for the first transition time $T$. The first transition will occur after an exponential time with rate parameter $\L + 2m + 2$. We need to carefully take note of which particles die in each of the six processes given in the order of \Ref{six}.
\vspace{-0.2cm}

Case 1: With probability $\frac{m}{\L + 2m + 2}$,
$U$ dies, $\b$ dies,
$U$ dies, $\a$ dies,
$U$ dies,  $\b$ dies, where $U$ is chosen uniformly from the points in $\xi$.

Case 2: With probability $\frac{1}{\L + 2m + 2}$,
$\a$ dies,  $\a$ dies,
$\a$ dies,  $\a$ dies,
$\b$ dies,  nothing dies.

Case 3: With probability $\frac{1}{\L + 2m + 2}$,
$\b$ dies, $\b$ dies,
nothing dies, nothing dies,
nothing dies, $\b$ dies.

Case 4: With probability $\frac{m}{\L + 2m + 2}$,
nothing dies, $\a$ dies,
nothing dies, nothing dies,
nothing dies, nothing dies.

And finally with probability $\frac{\L}{\L + 2m + 2}$ an immigration occurs at the same location for all the processes. If the first transition is an immigration step, then we can couple each successive pair of processes (in the order of \Ref{six}) exactly, until the process $\zxabt\dot$ returns to a state with $m+2$ particles. Furthermore, for points that exist in more than one pair of processes, they can also be coupled across the pairs. For example, the point at $\a$ exists in the first, second, third and fourth processes, so we couple the death time of $\a$ to be the same for all four processes. When $|\zxabt\dot|$ first returns to a state with $m+2$ particles we need only check if the points at $\a$ and $\b$ have died or not. If one of them has died, then the integrand in \Ref{six} becomes zero immediately. 

Simple calculations show that case 2 and case 3 cancel each other out exactly, upon some further simplification, and noting that the integrand of $I_3$ contributes nothing until the first transition each pair of processes have the same configurations,
\begin{align*}
|I _3| \leq &\frac{m}{\L + 2m + 2} \left[3 |I_1| + \| \Delta h^\resm \|\right] +  \frac{\L}{\L + 2m + 2} p^{(2)}_{\L, m+2} |I_3|,
\end{align*}
where $p^{(2)}_{\L, m+2}$ represents the probability that given the first transition was an immigration step, the points $\a$ and $\b$ are both still alive upon the first return to $m+2$ particles for the process $\zxabt\dot$. Noting that $p^{(2)}_{\L, m+2} \leq p_{\L, m+2}$ from Lemma~\ref{alphadeath} as the event that both the points at $\a$ and $\b$ survive is a subset of the event that the point at $\a$, survives $p^{(2)}_{\L, m+2} \leq \min\{ \frac{m+1}{m+3},\frac{m+2}{\L}\}$ (the $\frac{m+1}{m+3}$ comes from the fact that both points need to survive at least once death event), the bounds in Lemmas \ref{lemmaI1} and \ref{firstdiff} give the final result.
\end{proof}
We now have everything we need to complete the proof of Theorem~\ref{seconddiff}.
\begin{proof}[Proof of Theorem~\ref{seconddiff}]
The first set of bounds are simply twice the bounds for the first difference given in Theorem~\ref{firstdiff}. 

For the remaining bounds, recall \Ref{fourparts}. 
Lemma~\ref{lemmaI3} gives a a bound for the first integral. We now need only examine the last integral. 
\begin{align}
&\left|\int_0^\infty \E \left[ f(\zx + \d_\a \ind_{\t_\a > t} + \d_\b \ind_{\t_\b > t}) - f(\zx + \d_\a \ind_{\t_\a > t})\right.\right.\notag \\
	&\left.\phantom{VVVVV}- \left. f(\zx + \d_\b \ind_{\t_\b > t}) + f(\zx) \right]dt\right|\notag\\
	&=\left| \int_0^\infty \E \left[ f(\zx + \d_\a + \d_\b) - f(\zx + \d_\a)\right. \right.\notag\\
	&\left. \left. \phantom{VVVVV}- f(\zx + \d_\b)+ f(\zx)\middle| \tau_\a > t, \tau_\b > t \right]\Pr(\tau_\a > t, \tau_\b > t) dt\right|\notag\\
	&\leq \int_0^\infty e^{-2t} \E \left[ \frac{1}{|\zx| + 2} + \frac{1}{|\zx|+1} \right] dt\label{constantone}\\
	&\leq \int_0^\infty e^{-2t} \E \left[ \frac{1}{|Z_\xi(t)| + 2} + \frac{1}{|Z_\xi(t)| + 1} \right] \label{parttwo}dt,
\end{align}
where in the first inequality we have used that the conditional independence of $\tau_\a$ and $\tau_\b$ given the natural filtration of $\zxt\dot$,
\[ \Pr\left(\tau_\a>t, \tau_\b > t \left| |\zxt\dot|\right)\right. = \exp \left\{-\int_0^t (2 + 2 m \ind_{|\zxt(s)| = m}) ds \right\} \leq e^{-2t},\]
and for the second inequality we have used Lemma~\ref{zmz0}.

We can bound \Ref{constantone} by
\begin{align*}
\int_0^\infty& e^{-2t} \E \left[ \frac{1}{|\zx| + 2} + \frac{1}{|\zx|+1} \right]dt \\
	&\leq \int_0^\infty e^{-2t}\left[ \frac{1}{m+2} + \frac{1}{m+1}\right] dt\\
	&= \frac{2m+3}{(m+1)(m+2)} \cdot \frac{1}{2} < \frac{1}{m+1}.
\end{align*}
Schuhmacher \& Xia~\cite{SX08} bound the quantity in \Ref{parttwo} with ${\frac{2\log \L}{\L}}$ when ${\L\geq 1.76}$. Recalling our definition
\[ K_2 := \frac{2\log \L}{\L} \ind_{\L \geq 1.76} + \frac{1}{m+1} \ind_{\L < 1.76}. \]
We therefore have,
\begin{align}
\|\D^2 h^\resm\| \leq |I_3| + K_2,\label{D2parts}
\end{align}
and the bound now follows by using the bound from Lemma~\ref{lemmaI3}.
\end{proof}

An interesting question is whether there are equivalent statements to Remarks~\ref{sameasuncond} and \ref{sameasuncond2} for the second difference of $h$. 

If $m=0$, then $\tau_\a$ and $\tau_\b$ are independent exponential random variables, and hence it is easily seen that $|I_3|=0$. Therefore, this approach is consistent with the unconditional bounds of Schuhmacher \& Xia~\cite{SX08}. 

The answer to the question of whether there exists a $\L$-independent bound for $\|\D h^\resm \|$ is the same as for the first difference. Using the same arguments as in Remark~\ref{sameasuncond2} we can see that there appears to always be a need for a $\L$-dependent component.

%
%
\subsection{Non-uniform bounds for Stein factors}
The bounds in Theorems~\ref{firstdiff} and \ref{seconddiff} are both uniform in the choice of $\xi$ and are of order $\frac{\log(\L)}{\L}$, where this term comes from the unconditional bounds derived in Schuhmacher \& Xia~\cite{SX08}. Using a counterexample based on Brown \& Xia~\cite{BX95}, Schuhmacher \& Xia~\cite{SX08} show that the logarithmic terms in these bounds can not be removed if we want to use uniform Stein factors. In Brown, Weinberg \& Xia~\cite{BWX00}, the logarithmic terms are removed from Stein factors in the $d_2$ metric by allowing the Stein factors to rely upon the configurations involved, and alternate upper bounds can be derived. This argument is later simplified into a more elegant result in Brown \& Xia~\cite{BX01} and non-uniform bounds in the $\bar{d}_2$ metric are also given in Schuhmacher \& Xia~\cite{SX08}. In this subsection we will also derive non-uniform bounds for our Stein factors. Fortunately, given the lemmas we have already proven, the results are not difficult to arrive at.

Firstly, we will require the following non-uniform unconditional bounds.
\begin{lemma}[Schuhmacher \& Xia~\cite{SX08}]\label{L1L2}
\begin{align*}
\int_0^\infty e^{-t} \E \left[ \frac{1}{|Z_\xi(t)|+1} \right]dt \leq L_1 := \frac{1- e^{-(|\xi| \wedge \L)}}{|\xi| \wedge \L}.
\end{align*}
\begin{align*}
\int_0^\infty e^{-2t} \E \left[ \frac{1}{|Z_\xi(t)| + 2} + \frac{1}{|Z_\xi(t)| + 1} \right] dt \leq L_2 := \min\left\{ \frac{1}{|\xi| \wedge \L}, \frac{1.09}{|\xi| + 1} + \frac{1}{\L} \right\}.
\end{align*}
\end{lemma}
\begin{corollary}
For $m \geq 1$,
\begin{align*}
\|\D h^\resm_f(\xi;\a)\| \leq \frac{m+1}{|\xi| + 1} \left( \frac{1}{\L} + m L_1 \right) + L_1,
\end{align*}
and if $\L > m+2$
\begin{align*}
\|\D h^\resm_f(\xi;\a)\| \leq \frac{m+1}{|\xi|+1} \left( \frac{1}{\L (\L - m)} + \frac{m}{\L-m} L_1 \right) + L_1.
\end{align*}
For the second difference,
\begin{align*}
\|\D^2 & h_f^\resm(\xi;\a,\b)\| \leq \min\Bigg\{ \frac{2m+2}{|\xi| + 1} \left( \frac{1}{\L} + m L_1 \right) + 2L_1,\\
	&\frac{(m+2)(m+1)}{(|\xi|+2)(|\xi|+1)} \left[\frac{(4m+3)(m+3)}{(m+3)(2m+2) + 2\lambda} + \frac{4m(m+1)(m+3)}{(m+3)(2m+2)\lambda + 2\lambda^2} L_1\right] + L_2\Bigg\}.
\end{align*}
Furthermore, if $\L > m+2$,
\begin{align*}
\|\D^2 h_f^\resm(\xi;\a,\b)\| &\leq \min\Bigg\{ \frac{2m+2}{|\xi|+1} \left( \frac{1}{\L (\L - m)} + \frac{m}{\L-m} L_1 \right) + 2L_1,\\
&\ \ \ \ \frac{(m+2)(m+1)}{(|\xi|+2)(|\xi|+1)} \left[ \frac{3\L + m}{\L(\L-m)(\L+m)} + \frac{4\L m}{(\L-m)(\L+m)}L_1 \right]+ L_2\Bigg\}.
\end{align*}
\end{corollary}

\begin{proof}
We again use our approximate couplings as before, and define $\tau_\a$ as in \Ref{taua}. Then
\begin{align*}
|\D h^\resm_f(\xi;\a)| &\leq \left| \int_0^\infty \E \left[ f(\zxa) - f(\zx + \d_\a \ind_{\tau_\a > t}) \right] dt\right|\\
	& \ \ \ + \left| \int_0^\infty \E \left[ f(\zx + \d_\a \ind_{\tau_\a > t}) - f(\zx) \right] dt \right|.
\end{align*}
Following the same argument as in the proof of Theorem~\ref{firstdiff}, the second integral can be bound by $L_1$. 

For the first integral, recall that in the proof of Lemma~\ref{lemmaI1}, we used the `worst case scenario' of $|\xi| = m$. To achieve a non-uniform bound, we simply relax that restriction. Consider $|I_2(\xi)|$ where $|\xi|=k > m$. The processes only diverge in the manner described in the proof of Lemma~\ref{lemmaI2} if the particle at $\a$ is alive upon reaching a state where there are $m+1$ particles in the system. Therefore
\begin{align*}
|I_2(\xi)| &= \left| \int_0^\infty \E \left[ f(\zxa) - f(\zx + \d_\a \ind_{\tau_\a > t}) \right] dt \right|\\
	&\leq p_{\L,k} \cdot p_{\L,k-1} \cdot \ldots \cdot p_{\L,m+1} |I_2|\\
	&\leq \frac{k}{k+1} \cdot \frac{k-1}{k} \cdot \ldots \cdot \frac{m+1}{m+2} |I_2|\\
	&= \frac{m+1}{k+1} |I_2|
\end{align*}
Note that we can use exactly the same proof for bounding $|I_2|$ as in Lemma~\ref{lemmaI2} but using our non-uniform bound $L_1$ instead of $K_1$. 
For $\| \D^2 h_f^\resm(\xi;\a,\b)\|$, the first bounds are simply twice the first difference for our non-uniform bounds. For the second bounds, similarly to the first difference we note that as in \Ref{D2parts},
\begin{align*}
\|\D^2h_f^\resm(\xi;\a,\b)\| \leq |I_3(\xi)| + L_2.
\end{align*}
We can then use the same argument that the processes only diverge if both particles at $\a$ and $\b$ are alive upon the first time the process $\zxabt\dot$ reaches a state with $m+2$ particles. Therefore if $|\xi| = k>m$,
\begin{align*}
|I_3(\xi)| &\leq p_{\L,k+1}^{(2)}\cdot p_{\L,k}^{(2)} \cdot \ldots \cdot p_{\L,m+2}^{(2)}|I_3|\\
	&\leq \frac{k}{k+2} \cdot \frac{k-1}{k+1} \cdot \ldots \cdot \frac{m+1}{m+3} |I_3|\\
	&= \frac{(m+2)(m+1)}{(k+2)(k+1)}|I_3|.
\end{align*}
After substituting $L_1$ and $L_2$ for $K_1$ and $K_2$ in the bounds from Theorem~\ref{seconddiff} we then get the final results.
\end{proof}
\begin{remark}
Note that we use the somewhat crude bound of $p_{\L,k} \cdot p_{\L,k-1} \cdot \ldots \cdot p_{\L,m+1} \leq \frac{m+1}{k+1}$ in the bound for $|I_1(\xi)|$ (and the similar bound for $|I_3(\xi)|$) purely for simplicity. Using Lemma~\ref{alphadeath} appropriately, a sharper $\L$-dependent bound could also be devised.
\end{remark}
\subsection{Conditional Bernoulli process approximation}
In this section we give a simple example of conditional Poisson point process approximation using Stein's method.

Let $\G= [0,1]$ with metric $d_0(x,y) = |x-y|$, and let $X_1, \ldots, X_n$ be i.i.d. Bernoulli random variables with common parameter $p$. $\Xi = \sum_{i=1}^n X_i \d_{i/n}$ defines a \emph{Bernoulli process}.

Let $T_1, \ldots, T_n$ be i.i.d.\ uniform random variables on $[0,1]$ which are independent of the $X_i$'s. $W = \sum_{i=1}^n X_i \d_{T_i}$ defines a \emph{binomial process}.

We seek to approximate a conditional Bernoulli process $\Xi^\resone$ with a conditional Poisson process $\Po^\resone(\Lb)$, conditioning upon at least $1$ atom. To achieve this, we will first compare the conditional Bernoulli process $\Xi^\resone$ with the conditional binomial process $W^\resone$, and then compare the conditional binomial process with the conditional Poisson process $\Po^\resone(\Lb)$. 
\begin{theorem}
Using the setup above, if we set $\Lb(dx) = np\ dx$,
\begin{align}
\bar{d}_2(\cL(\Xi^\resone), \Po^\resone(\Lb)) &\leq \frac{1}{1-(1-p)^n} \cdot \Bigg[ \left( \frac{1}{2n} + \frac{p}{2} \right) \wedge \frac{1}{\sqrt{3np}} \notag\\
	&\ \ \ + \frac{p + 2p(0.95 + \log^+(np))}{(0.5 \vee \sqrt{(n-1)p(1-p)})}\Bigg],\label{bound1}
\end{align}
and if $\L > 3$,
\begin{align}
\bar{d}_2(\cL(\Xi^\resone), \Po^\resone(\Lb)) &\leq \frac{1}{1-(1-p)^n} \cdot \Bigg[ \left( \frac{1}{2n} + \frac{p}{2} \right) \wedge \frac{1}{\sqrt{3np}} \notag\\
	&\ \ \ + \frac{p + np^2(0.95 + \log(np))}{(np-1)\cdot (0.5 \vee \sqrt{(n-1)p(1-p)})}\Bigg].\label{bound2}
\end{align}
\end{theorem}
\begin{proof}
Xia \& Zhang~\cite{XZ08} have shown that,
\[d_2(\cL(\Xi), \cL(W)) \leq \left(\frac{1}{2n} + \frac{p}{2} \right) \wedge \frac{1}{\sqrt{3np}}.\]
Noting that $|\Xi| = |W|$, $\Pr(|\Xi| \geq 1) = \Pr(|W| \geq 1)$, and $d_2$ and $\bar{d}_2$ are the same given both configurations have the same number of particles, hence
\[\bar{d}_2(\cL(\Xi^\resone), \cL(W^\resone)) \leq \frac{1}{1-(1-p)^n} \left[\left(\frac{1}{2n} + \frac{p}{2} \right) \wedge \frac{1}{\sqrt{3np}}\right].\]
We now need to find a bound for $\bar{d}_2(\cL(W^\resone), \Po^\resone(\Lb))$, where we set $\Lb(dx) = np\ dx$. To this end, we need to find a bound for $|\E \cA^\resone h(W^\resone)|$, where
\begin{align}
\E \cA^\resone h(W^\resone) &= \E \left\{\int_0^1 \left[ h(W^\resone + \d_\a) - h(W^\resone)\right] \Lb(d\a)\right\}\notag\\
	&\ \ - \E \left\{\int_0^1 \left[ h(W^\resone) - h(W^\resone - \d_\a) \right] W^\resone(d\a) \ind_{|W^\resone| \geq 2}\right\}.\label{CPPex1}
\end{align}
We re-write the first term in \Ref{CPPex1} as
\begin{align}
&\E \left\{\int_0^1 \left[ h(W^\resone + \d_\a) - h(W^\resone)\right] \Lb(d\a)\right\}\notag\\
 &= \frac{p}{\Pr(|W| \geq 1)} \sum_{i=1}^n  \E \left\{ \left[ h(W + \d_{S_i}) - h(W) \right] \ind_{|W| \geq 1} \right\}\label{CPPex2},
\end{align}
where the $S_i$ are i.i.d.\ uniform random variables on $[0,1]$. For the second term of \Ref{CPPex1},
\begin{align}
&\E \left\{ \int_0^1 \left[ h(W^\resone) - h(W^\resone - \d_\a) \right] W^\resone(d\a) \ind_{|W^\resone| \geq 2}\right\}\notag\\
	&\ \  =\frac{1}{\Pr(|W| \geq 1)} \sum_{i=1}^n \E \left\{ \left[ h(W) - h(W-\d_{T_i}) \right] X_i \ind_{|W| \geq 2} \right\}\notag\\
	&\ \ = \frac{p}{\Pr(|W| \geq 1)} \sum_{i=1}^n \E \left\{ \left[ h(W^i + \d_{T_i}) - h(W^i) \right] \ind_{|W^i| \geq 1} \right\}\label{CPPex3},
\end{align}
where $W^i = W - X_i\d_{T_i}$, and the last equality is true as $W$ and $W - \d_{T_i}$ are only different if $X_i=1$ which occurs with probability $p$. Note that since $W^i$ and $T_i$ are independent, we can replace $T_i$ with $S_i$ without changing the value of the expectation. We can now combine~\Ref{CPPex1}, \Ref{CPPex2} and \Ref{CPPex3} to give
\begin{align}
\left|\E\cA^\resone h(W^\resone)\right|  &= \frac{p}{\Pr(|W| \geq 1)} \Bigg|\sum_{i=1}^n \E \Big[ (h(W + \d_{S_i}) - h(W))\ind_{|W| \geq 1}\notag\\
	&\ \ \ \  - (h(W^i + \d_{S_i}) - h(W^i)) \ind_{|W^i| \geq 1} \Big] \Bigg|\label{CPPex4}\\
	&=\frac{p^2}{\Pr(|W| \geq 1)} \Bigg|\sum_{i=1}^n \E \Big[ (h(W^i + \d_{T_i} + \d_{S_i}) - h(W^i + \d_{T_i}))\notag\\
	&\ \ \ \  - (h(W^i + \d_{S_i}) - h(W^i)) \ind_{|W^i| \geq 1} \Big] \Bigg|\label{tada}\\
	&\leq \frac{np^2}{\Pr(|W| \geq 1)} \| \D^2 h^\resone \|,\notag
\end{align}
where in the last equality we used the fact that the summand is non-zero only if $X_1 = 1$, which occurs with probability $p$, and the indicator variable in \Ref{tada} can be dropped off by noting that we can arbitrarily set $h(\emptyset) = \E h(\d_{S_i})$.

As an alternative to this bound, we can utilise the approach in Section~4.2 of Schuhmacher \& Xia~\cite{SX08}, and use bounds of the first difference of $h$. Following similar arguments to before, using \Ref{CPPex4} we can show that
\begin{align*}
\left|\E\cA^\resone h(W^\resone)\right|  &= \frac{np}{\Pr(|W| \geq 1)} \Bigg| \E \Big[ (h(W + \d_{S_0}) - h(W))\ind_{|W| \geq 1}\\
	&\ \ \ \  - (h(W^1 + \d_{S_0}) - h(W^1)) \ind_{|W^1| \geq 1} \Big] \Bigg|,
\end{align*}
where $S_0$ independent of $W$ and is uniform on $[0,1]$. Define
\[ g(i) = \E\big[ h(W + \d_{S_0}) - h(W) \big| |W| = i\big] = \E \left[ h\left(\sum_{j=1}^i \d_{T_j}+ \d_{S_0} \right) - h\left(\sum_{j=1}^i \d_{T_j} \right)\right].\]
Then
\begin{align}
|\E \cA^\resone h(W^\resone)| &= \frac{np}{\Pr(|W| \geq 1)} \left|\E [ g(|W|) - g(|W^1|) ] \right|\notag\\
	&= \frac{np^2}{\Pr(|W| \geq 1)} \left|\E [g(|W^1| + 1) - g(|W^1|) ] \right|\notag\\
	&\leq \frac{2np^2}{\Pr(|W| \geq 1)}\| g \| d_{TV}(\cL(|W^1|+1), \cL(|W^1|)),\label{CPPex5}
\end{align}
as similarly to earlier $|W| = |W^1+1|$ with probability $p$, and noting that for any two non-negative integer random variables $Z_1, Z_2$ (see Appendix A.1 of Barbour, Holst \& Janson~\cite{BHJ}),
\[ d_{TV}(\cL(Z_1), \cL(Z_2)) = \frac{1}{2} \sup_{f: \Z_+ \to [-1,1]} | \E f(Z_1) - \E f(Z_2)|.\]
Furthermore, from Lemma 1 of Barbour \& Jensen~\cite{BJ89},
\begin{align}d_{TV}(\cL(|W^1|+1), \cL(|W^1|)) \leq 1 \wedge \frac{1}{2\sqrt{(n-1)p(1-p)}}.\label{CPPex6}\end{align}
Noting that $\|g\| \leq \|\D h^\resone \|$, \Ref{CPPsteineq}, \Ref{CPPex5} and \Ref{CPPex6} imply
\begin{align*}
\bar{d}_2(\cL(W^\resone), \Po^\resone(\Lb)) &= \sup_f| \E \cA h_f(W)|\\
	 &\leq \frac{\| \D h^\resone\| \cdot np^2}{(1-(1-p)^n) \cdot (0.5 \vee \sqrt{(n-1)p(1-p)})}.
\end{align*}
Hence if $\L = np > 3$,
\begin{align*}
\bar{d}_2(\cL(W^\resone), \Po^\resone(\Lb)) \leq \frac{p + np^2(0.95 + \log(np))}{(1-(1-p)^n) (np-1)\cdot (0.5 \vee \sqrt{(n-1)p(1-p)})},
\end{align*}
otherwise,
\begin{align*}
\bar{d}_2(\cL(W^\resone), \Po^\resone(\Lb)) \leq \frac{p + 2p(0.95 + \log^+(np))}{(1-(1-p)^n) \cdot (0.5 \vee \sqrt{(n-1)p(1-p)})}.
\end{align*}
\end{proof}
We assess these bounds under two scenarios. First, when $n$ is fixed and $p \to 0$. In this scenario we would use \Ref{bound1} as our bound. The bound appears to not be particularly good due to the term $\left( \frac{1}{2n} + \frac{p}{2} \right) \wedge \frac{1}{\sqrt{3np}}$. This is because this bound was derived by Xia \& Zhang~\cite{XZ08} with the intention to be used for large $n$. There exists a possibility that the bound \Ref{bound1} can be improved for this scenario by improving the unconditional bound for $d_2(\cL(\Xi), \cL(W))$.

In the case where $p$ is fixed and $n \to \infty$, we use the bound in \Ref{bound2}, and this bound appears to be quite good. The bound is of order 
\[\frac{\sqrt{p}(0.95 + \log(np))}{\sqrt{n(1-p)}},\]
which is asymptotically equivalent to the unconditional bounds. Considering Remark~\ref{sameasuncond2}, this is what we would expect.
\begin{remark}
It would undoubtedly be nice to have an example of a Hawkes point process. However the inherent `independent increment' nature of the Poisson process would make this unsuitable for any non-trivial Hawkes point processes, as a Hawkes point process usually will contain clustering. A more natural choice of point process for approximation would be a conditional compound Poisson point process. This paper is intended as a first step into understanding how one can manipulate generators to approximate conditional point processes. It remains to be seen if it can be generalised to a wider class of point processes by applying different conditions to the immigration or death processes. 
\end{remark}
\newpage
\bibliographystyle{plain}
\bibliography{references}{}
\end{document}